\documentclass{amsart}

\usepackage[latin1]{inputenc}
\usepackage{color}
\usepackage[all,dvips]{xy}
\usepackage[dvips]{graphicx}

\newtheorem{theorem}{Theorem}[section]
\newtheorem{lemma}[theorem]{Lemma}
\newtheorem{proposition}[theorem]{Proposition}

\theoremstyle{definition}

\theoremstyle{remark}
\newtheorem{remark}[theorem]{Remark}

\numberwithin{equation}{section}

\newcommand{\Q}{{\mathbb{Q}}}

\newcommand{\Z}{{\mathbb{Z}}}

\newcommand{\cO}{{\mathcal{O}}}

\newcommand{\cS}{{\mathcal{S}}}

\DeclareMathOperator{\Sel}{Sel}
\DeclareMathOperator{\rank}{rank}

\DeclareMathOperator{\Jac}{Jac}

\begin{document}

\title[Covering techniques and rational points on some genus 5 curves]{Covering techniques and rational points\\ on some genus 5 curves}
\author{Enrique Gonz\'alez-Jim\'enez}
\address{Universidad Aut{\'o}noma de Madrid, Departamento de Matem{\'a}ticas and Instituto de Ciencias Matem{\'a}ticas (ICMat), 28049 Madrid, Spain}
\email{enrique.gonzalez.jimenez@uam.es}
\thanks{The author was supported in part by grant MTM2012--35849.}

\date{\today}


\begin{abstract}
We describe a method that allows, under some hypotheses, computation of all the rational points of some genus $5$ curves defined over a number field. This method is used to solve some arithmetic problems that remained open.
\end{abstract}
\maketitle

\section{Introduction}
Several arithmetic problems are parametrized by the rational points of a curve over a number field $K$. In the cases where there are only squares involved, sometimes these curves may be written as the intersection of diagonal quadrics (only squares of the variables appear) in some projective space. The easiest case we are interested in is $C:aX_0^2+bX_1^2=X_2^2$, that represents a conic in $\mathbb P^3$. This case is well-understood and there are good algorithms that describe when there is a solution and, in that case, find them all. A next case is $C:\{aX_0^2+bX_1^2=X_2^2,cX_0^2+dX_1^2=X_3^2\}$, which represents a genus $1$ curve (if $ad-bc\ne 0$) in $\mathbb P^4$. Although, nowadays there is not a deterministic algorithm to determine if $C(K)$ is empty and/or to compute $C(K)$, it has been deeply studied. Finally, we have the case $C:\{aX_0^2+bX_1^2=X_2^2,cX_0^2+dX_1^2=X_3^2,eX_0^2+fX_1^2=X_4^2\}$. This curve is generically of genus $5$ and there are not known algorithms to compute $C(K)$. In this paper, our purpose is to give an algorithm to compute (under some hypotheses) $C(K)$. In fact, in section \ref{sec2}, we describe a more general algorithm to compute the rational points of some genus $5$ curve where the above curves are a particular case. This algorithm is based on some previous works with Xavier Xarles (for a single curve \cite{Gonzalez-Jimenez-Xarles2011} or for family of curves \cite{Gonzalez-Jimenez-Xarles2013b,Gonzalez-Jimenez2013}). 

In section \ref{sec3} we apply the algorithm described in section \ref{sec2} to some arithmetic problems that have remained open in the literature. At the end of the paper we include an appendix dedicated to quartic elliptic curves. There we show some results that will be useful for the use of the algorithm of section \ref{sec2}.

\section{An algorithm}\label{sec2}

Let $p_1,p_2$ be two coprime monic quartic separable polynomials with coefficients in a number field $K$. Consider the genus 5 curve $C$ defined in $\mathbb{A}^3$ by 
\begin{equation}\label{model4}
C\,:\,\{\,y_1^2=p_1(t)\,,\,y_2^2=p_2(t)\,\}.
\end{equation}
In this section we show an algorithm that allows (under some hypotheses) computation of $C(K)$. This method is based on the covering collections technique (cf. \cite{Coombes-Grant1989,Wetherell1997}) and the elliptic curve Chabauty method (cf. \cite{Flynn-Wetherell2001,Bruin2003}). 

Thanks to the shape of the curve $C$, it has two degree $2$ maps defined over $K$ to the genus $1$ curves given by the equations $F_i\,:\,y_i^2=p_i(t)$, for $i=1,2$. 

Now, consider a factorization of each polynomial $p_i(t)$ as product of two degree two polynomials $p_{i+}(t)$ and $p_{i-}(t)$ defined over an algebraic extension $L$ of $K$. Each factorization $p_i(t)=p_{i+}(t)p_{i-}(t)$ determines an unramified degree $2$ covering $\chi_i:F'_i\to F_i$ given by the curve
$$
F'_i\,:\,\{y_{i+}^2\!\!= p_{i+}(t)\,,\,y_{i-}^2\!\!= p_{i-}(t)\,\},
$$
and $\chi_i(t,y_{i+},\,y_{i-})=(t,y_{i+}y_{i-})$, for $i=1,2$. Thus, each covering corresponds to a degree $2$ isogeny $\phi_i:E_i'\to E_i$, where $E_i=\Jac(F_i)$ and $E_i'=\Jac(F_i')$.

Moreover, these factorizations together determine a Galois cover of $C$ with Galois group $(\Z/2\Z)^2$ that can be described as the curve in $\mathbb A^5$ given by
$$
D\,:\,\{y_{1+}^2\!\!= p_{1+}(t)\,,\,y_{1-}^2\!\!= p_{1-}(t)\,,\,y_{2+}^2\!\!= p_{2+}(t)\,,\,y_{2-}^2\!\!=  p_{2-}(t)\,\},
$$
which is a curve of genus $17$, along with the map $\chi:D\to C$ defined as
$$\chi(t,y_{1+},\,y_{1-},\,y_{2+},\,y_{2-})=(t,y_{1+}y_{1-},\,y_{2+}y_{2-}).$$
Now, for any pair $(\delta_1,\delta_2)\in K^2$ we define the twist $\chi^{(\delta_1,\delta_2)}:D^{(\delta_1,\delta_2)}\to C$ of the covering $\chi:D\to C$ by:
$$
D^{(\delta_1,\delta_2)}\,:\,\{\delta_1y_{1+}^2\!\!= p_{1+}(t)\,,\,\delta_1y_{1-}^2\!\!= p_{1-}(t)\,,\,\delta_2 y_{2+}^2\!\!= p_{2+}(t)\,,\,\delta_2 y_{2-}^2\!\!=  p_{2-}(t)\,\},
$$
and 
$$
\chi^{(\delta_1,\delta_2)}(t,y_{1+},\,y_{1-},\,y_{2+},\,y_{2-})=(t,\delta_1y_{1+}y_{1-},\,\delta_2 y_{2+}y_{2-}).
$$
Then, by a classical theorem of Chevalley and Weil \cite{Chevalley-Weil1932} we have
$$
C(K)\subseteq \bigcup_{\delta\in K^2} \chi^{\delta}(\{P\in D^{\delta}(L)\ : \ \chi^{\delta}(P)\in C(K)\}).
$$
Notice that only a finite number of twists have points locally everywhere, and these twists can be explicitly described. This finite set, that we denote by $\mathfrak{S} \subset (K^*)^2$, may be described, thanks to Proposition \ref{2cov}, in terms of a set $\mathcal{S}_L(\phi_i)$ of representatives in $L$ of the image of the Selmer groups of the degree $2$ isogenies $\phi_i:E_i'\to E_i$ in $L^*/(L^*)^2$ via the natural map, for $i=1,2$. That is, $\mathfrak{S} =\mathcal{S}_L(\phi_1)\times \mathcal{S}_L(\phi_2)$.

Once we have determined the finite set $\mathfrak{S} $, the next challenge is to compute all the points $P\in D^{\delta}(L)$ such that $\chi^{\delta}(P)\in C(K)$ for any $\delta\in \mathfrak{S}$. For this purpose, we are going to use the elliptic curve Chabauty method.  For $s=(s_1,s_2)\in\{\pm,\pm\}$ consider the quotient $\pi_s\,:\,D\to H_s$ where
$$
H_s\,:\, z^2=p_{1s_1}(t)p_{2s_2}(t)\quad\mbox{and}\quad \pi_s(t,y_{1+},\,y_{1-},\,y_{2+},\,y_{2-})=(t,y_{1s_1}y_{2s_2}).
$$
Then for any $\delta=(\delta_1,\delta_2)\in \mathfrak{S}$ we define $\pi^{\delta}_s\,:\,D\to H^{\delta}_s$ where
$$
H^{\delta}_s\,:\, \delta_1\delta_2 z^2=p_{1s_1}(t)p_{2s_2}(t)\quad\mbox{and}\quad \pi^{\delta}_s(t,y_{1+},\,y_{1-},\,y_{2+},\,y_{2-})=(x,y_{1s_1}y_{2s_2}).
$$
which, in fact, only depends on the product $\delta_1\delta_2$. Therefore, we can replace $\mathfrak{S}$ by:
$$ 
\mathfrak{S} =\{\delta_1\delta_2 \ : \ \delta_1\in{\cS_L}(\phi_{1}), \delta_2\in \cS_L(\phi_{2})\}.
$$

The following commutative diagram illustrates all the
curves and morphisms involved in our problem:
$$
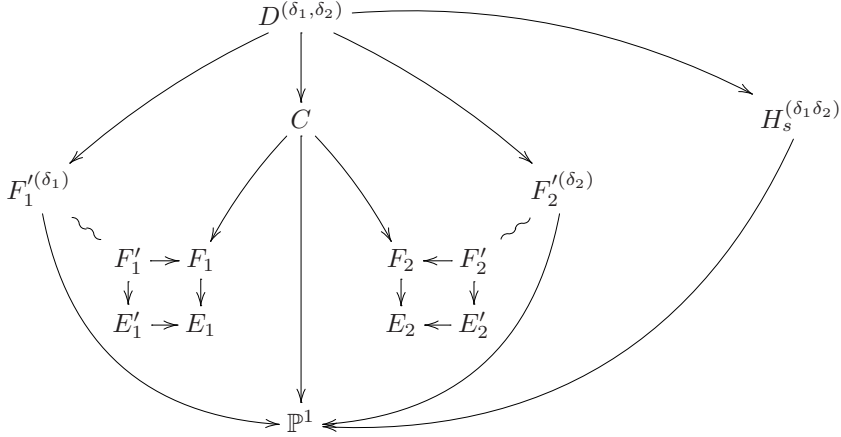
\xymatrix@R=0.7pc@C=0.8pc{
&             &            &      \ar@/_2mm/[dddlll] D^{(\delta_1,\delta_2)} \ar@{->}[dd]    \ar@/^2mm/[dddrrr]   \ar@/^6mm/[ddrrrrrrr] &           &                           &      &       &     &   & \\
&                             &            &       &           &                           &      &    &        &   \\
&                     &            &   \ar@/_1mm/[ddl]     C  \ar@{->}[ddddd]    \ar@/^1mm/[ddr]                               &           &                           &   &  & & & H_s^{(\delta_1\delta_2)} \ar@/^14mm/[dddddlllllll] \\
F_1'^{(\delta_1)}\ar@{~}[dr] \ar@/^-10mm/[ddddrrr]  &                                           &            &       &           &                          &   \ar@{~}[dl]  F_2'^{(\delta_2)} \ar@/^10mm/[ddddlll]  &       &        & & \\
& F_1'  \ar@{->}[d]  \ar@{->}[r]&  F_1 \ar@{->}[d]   &                                              &  F_2  \ar@{->}[d] &\ar@{->}[l] F_2' \ar@{->}[d] &  &  &  &&   & \\
& E_1'  \ar@{->}[r]&  E_1    &                                              &  E_2   &\ar@{->}[l] E_2'  &      &   & & & &\\
& &     &                                            &    &  &      &   & & & &\\
& &     &              \mathbb P^1                                &    &  &      &   & & & &\\
}
$$
Notice that, in the diagram above, all the morphisms to $\mathbb P^1$ are given by the parame\-ter $t$.

We have obtained for a fixed $\delta\in\mathfrak{S}$ and for any $s\in \{(\pm,\pm)\}$:
$$
\{t \in \Q |\  \exists Y\in L^4 \mbox{ with } (t,Y)\in D^{(\delta)}(L)\} \subseteq  \{t \in \Q |\  \exists z \in
L \mbox{ with }  (t,z) \in
H_s^{\delta}(L)\}.
$$

Then the algorithm works out if we are able to compute for any $\delta \in\mathfrak{S}$, all the
points $(t,z) \in H_{s}^{\delta}(L)$ with $t\in \mathbb P^1(\Q)$ for some choice of the signs $s\in \{(\pm,\pm)\}$.  This computation can be done in two steps as follows:

\noindent ($\mbox{1}^{\mbox{st}}$) We must determine if $H_{s}^{\delta}(L)$ is empty.  Bruin and
Stoll \cite{Bruin-Stoll2009} developed a (non-deterministic) method to determine if this happens.

\noindent ($\mbox{2}^{\mbox{nd}}$) In the case that $H_{s}^{\delta}(L)$ is non-empty, we use the elliptic curve
Chabauty technique (cf. \cite{Bruin2003}). To do that we must  compute if the rank of the Mordell-Weil group of ${H_{s}^{\delta}}(L)$ is less than
the degree of $L$ over $\Q$. We also need to determine a subgroup of finite index of this group to
carry out the elliptic curve Chabauty method.

In practice, we consider only the case $K=\Q$ and $L$ a quadratic number field, because the computation of the Mordell-Weil group of an elliptic curve over a number field of higher degree is too expensive computationally. We have  implemented the algorithm in \verb|Magma|  \cite{magma}.
\subsection{Diagonal genus $5$ curves}\label{sec_diagonal5}
Let $K$ be a number field and $a,b,c,d,e,f\in K$. Denote by $C$ the intersection of the following three quadrics in $\mathbb P^4$:
\begin{equation}\label{model}
C\,:\,\left\{\begin{array}{c}aX_0^2+bX_1^2=X_2^2\\ cX_0^2+dX_1^2=X_3^2\\ eX_0^2+fX_1^2=X_4^2\end{array}\right\}.
\end{equation}
Suppose that the three quadratic forms (in the variables $X_0$ and $X_1$) defining each quadric are non-singular and non-proportional. Then $C$ is a (non-singular) genus $5$ curve (cf. \cite{Bremner1997}). Note that any non-hyperelliptic genus $5$ curve may be given (after the canonical map in $\mathbb{P}^4$ and Petri's Theorem) as the intersection of three quadrics. That is the reason why this kind of genus 5 curve will be called diagonal by us. Moreover, the jacobian of $C$ is $K$-isogenous to the product of the following five elliptic curves (cf. \cite{Bremner1997}) :
$$
\begin{array}{l}
 {E}_4\,:\, y^2\,=\,x(x+a d)( x+ c b), \\
 {E}_3\,:\, y^2\,=\,x(x+a f)( x+ e b),   \\
 {E}_2\,:\, y^2\,=\,x(x+c f)( x+ e d),  \\
 {E}_1\,:\, y^2\,=\,x(x-d(af-eb))(x-f(ad-cb)), \\
 {E}_0\,:\, y^2\,=\,x(x+c(af-eb))( x+ e(ad-cb)).  
 \end{array}
$$
Note that $E_i$ is the jacobian of the genus $1$ curve obtained by removing the variable $X_i$ from the equation of $C$. Moreover, the isogeny between $\Jac(C)$ and  $E_0\times\cdots\times E_4$ comes from the forgetful maps $\pi_i:C\to  E_i$. 

We associate to model (\ref{model}) of the curve $C$ the following two matrices:
$$
\mathcal{M}_C=\begin{pmatrix}1 & 0 & 0 & -a & -b  \\ 0 & 1 & 0 & -c & -d \\  0 & 0 & 1 & -e& -f\end{pmatrix}\quad\quad\mbox{and}\qquad\quad
\mathcal{R}_C=\begin{pmatrix}a & b \\c & d\\ e& f \end{pmatrix}.
$$
We call $\mathcal{M}_C$ (resp. $\mathcal{R}_C)$ the matrix (rep. reduced matrix) of the model  (\ref{model}). Notice that if we permute the columns of $\mathcal{M}_C$ then the echelon form of this new matrix give us a new matrix and a new reduced matrix of a new model of $C$ (as the intersection of three quadrics in $\mathbb P^4$). That is, there are ten ways to write the diagonal genus $5$ curve as the intersection of three diagonal quadrics in $\mathbb P^4$.

Let us give a new model of the diagonal genus $5$ curve similar to the one given by (\ref{model4}). Suppose that $[x_0:x_1:x_2:x_3:x_4]\in C(K)$. Then the techniques developed in section \ref{sec_diagona_genus1} allow us to determine two coprime monic quartic separable polynomials with coefficients in $K$ associated to the matrices:
$$
\mathcal{R}_3=\begin{pmatrix}a & b \\c & d \end{pmatrix}\quad\quad\mbox{and}\qquad\quad
\mathcal{R}_4=\begin{pmatrix}a & b \\ e& f \end{pmatrix}.
$$
That is, $p_3=p_{\mathcal{R}_3}$ and $p_4=p_{\mathcal{R}_4}$ (see equation (\ref{eq_quartic_diagonal_genus1}) in section \ref{sec_diagona_genus1}). These polynomials define the following new model of the diagonal genus $5$ curve $C$:
$$
C\,:\,\{\,y_3^2=p_3(t)\,,\,y_4^2=p_4(t)\,\}.
$$
The change of model is obtained by parametrizing the conic $aX_0^2+bX_1^2=X_2^2$  by the point $[x_0:x_1:x_2:x_3:x_4]$ and it is given by:
$$
[X_0:X_1:X_2:X_3:X_4]\longmapsto (t,y_3,y_4)=\left(\frac{b (x_1 X_2 - X_1 x_2) x_3^2}{x_0 X_2 - X_0 x_2},x_3X_3,x_4X_4\right).
$$
Moreover, for $i=3,4$, $E_i$ is the jacobian of the quartic genus $1$ curve defined by $y_i^2=p_i(t)$.

Now, to apply the algorithm described in section \ref{sec2} we need factorizations of the quartic polynomials $p_3$ and $p_4$. These have been given at section \ref{sec_diagona_genus1}. In particular, for $i\in\{3,4\}$, we have three factorizations $p_i(t)=p_{i,j_i,+}(t)p_{i,j_i,-}(t)$, $j_i\in\{1,2,3\}$, over the field $K(\alpha_{i,j_i})$ where:
$$
\begin{array}{lcclccl}
\alpha_{3,1}=\sqrt{-cd}\,, &&  \alpha_{3,2}=\sqrt{-c(ad-bc)}\,, && \alpha_{3,3}=\sqrt{d(ad-bc)}\,,\\
\alpha_{4,1}=\sqrt{-ef}\,,  && \alpha_{4,2}=\sqrt{-e(af-be)}\,, && \alpha_{4,3}=\sqrt{f(af-be)}\,.
\end{array}
$$
Each factorization $(i,j_i)$ corresponds to the following $2$-torsion point on $E_i(K)$: 
$$
\begin{array}{l}
P_{3,1}=(0,0),\qquad P_{3,2}=(-b\,c,0),\qquad P_{3,3}=(-a\,d,0)\,,\\
P_{4,1}=(0,0),\qquad P_{4,2}=(-b\,e,0),\qquad P_{4,3}=(-a\,f,0)\,.
\end{array}
$$
And each two torsion gives a $2$-isogeny $\phi_{i,j_i}:E_i\to E'_i$.  

Moreover, thanks to the shape of the diagonal genus $5$ curves, we have that the number of twists to be checked may be smaller than expected (see \cite[Lemma 16]{Gonzalez-Jimenez-Xarles2013b}). Let $\Upsilon$ be the group of automorphisms  of the curve $C$ generated by the automorphisms $\tau_i(X_i)=-X_i$ and
$\tau_i(X_j)=X_j$ if $i\ne j$, for $i=0,1,2,3,4$. Fix  $j_3,j_4\in\{1,2,3\}$. Consider $L=K(\alpha_{3,j_3},\alpha_{4,j_4})$ and denote by $\widetilde{\cS_L}(\phi_{3,j_3})$ a set of representatives of $\Sel(\phi_{3,j_3})$ modulo the subgroup generated by the image of the trivial points $[\pm x_0:\pm x_1:\pm x_2:\pm x_3:\pm x_4]$ in this Selmer group. Consider the subset $\widetilde{\mathfrak{S}} \subset K^*$ defined by 
$$ 
\widetilde{\mathfrak{S}} =\{\delta_3\delta_4 \ : \ \delta_3\in\widetilde{\cS_L}(\phi_{3,j_3}), \delta_4\in \cS_L(\phi_{4,j_4})\}.
$$

The method allows us to compute $C(K)$ if we are able to calculate, for
some choice of $j_3,j_4\in\{1,2,3\}$, and for any $\delta \in\widetilde{\mathfrak{S}}$, all the
points $(t,w) \in
H_{s}^{\delta}(K(\alpha_{1,j_1},\alpha_{2,j_2}))$ with $t\in\mathbb P^1(K)$ for some choice of the signs $s\in\{(\pm,\pm)\}$.

Hence we have $60$ possible choices of $\mathcal{R}_C$, $j_3$ and $j_4$, and we
need to find one of them where we can carry out these computations
for all the elements $\delta\in\widetilde{ \mathfrak{S}}$.

\section{Examples}\label{sec3}
In this section we are going to characterize the solutions of some arithmetic problems in terms of the rational points of some genus $5$ curves. Then we will solve these problems by computing all the rational points of such curves using the algorithm described in section \ref{sec2}.

\subsection{Arithmetic progressions on Pell equations}

Let $Y_n=a+(n-1) q$, $n=1,\dots,5$ with $a,q\in\Q$ be the $Y$-coordinates of the solutions $(X_n,Y_n)$ , $n=1,\dots,5$, to the Pell equation $X^2-dY^2 = m$. Then we say that  $(X_n,Y_n)$  (or just $Y_n$), $n=1,\dots,5$, are in arithmetic progression on the curve $X^2-dY^2 = m$. Following Pethö and Ziegler \cite{Petho-Ziegler2008}, one can obtain the system of $5$ equations:
$$
\begin{array}{c}
X_1^2-da^2=m\,,\quad X_2^2-d(a+q)^2=m\,,\quad X_3^2-d(a+2q)^2=m\,,\\[2mm] X_4^2-d(a+3q)^2=m\,,\qquad X_5^2-d(a+4q)^2=m\,.
\end{array}
$$
Eliminating $m$ we obtain an equivalent system of 4 equations:
$$
\begin{array}{ccc}
X_2^2-X_1^2 =dq(2a+q), & &	X_3^2-X_2^2 =dq(2a+3q),\\[2mm]
X_4^2-X_3^2 =dq(2a+5q), & &	X_5^2-X_4^2 =dq(2a+7q),
\end{array}
$$
and eliminating $d$:$$
C_{a,q}:\left\{
\begin{array}{l}
X_2^2(4a + 4q) = X_1^2(2a + 3q) + X_3^2(2a + q)\\[2mm]
X_3^2(4a + 8q) = X_2^2(2a + 5q) + X_4^2(2a + 3q)\\[2mm]
X_4^2(4a + 12q) = X_3^2(2a + 7q) + X_5^2(2a + 5q)
\end{array}
\right\}.
$$
Therefore the matrix corresponding to the variables $X_1^2,\dots,X_5^2$ is
$$
\widehat{\mathcal{M}}_{C_{a,q}}=
\begin{pmatrix}
2a+3q & -4(a+q) & 2a+q & 0 & 0\\
0 & 2a+5q & -4(a+2q) & 2a+3q & 0 \\
0 & 0 & 2a+7q & -4(a+3q) & 2a+5q
\end{pmatrix}.
$$
Notice that the points $[\pm 1:\pm 1:\pm 1:\pm 1:\pm 1]\in C(\Q)$ correspond to $(d,m)=(0,1)$.

Pethö and Ziegler \cite[\S 8. Open questions]{Petho-Ziegler2008} asked the following:

\

{\it \underline{Question}:} {\it \lq\lq Can one prove or disprove that there are $d$ and $m$ with $d > 0$ and not a perfect square such that $Y = 1, 3, 5, 7, 9$ are in arithmetic progression on the curve $X^2-dY^2 = m$?" }

\

In this section our target is to answer the question above. Then, if we are looking for $d$ and $m$ such that $Y = 1, 3, 5, 7, 9$ is an arithmetic progression on the curve $X^2-dY^2 = m$ then we have $a=1$ and $q=2$. In particular, it may be proved that $C:=C_{1,2}$ is a diagonal genus $5$ curve just computing the matrix associated to a model of the form (\ref{model}) coming from the matrix $\widehat{\mathcal{M}}_C$. That is: 
$$
\widehat{\mathcal{M}}_C=\begin{pmatrix}
8 & -12 & 4 & 0 & 0\\
0 & 12 & -20 & 8 & 0 \\
0 & 0 & 16 & -28 & 12
\end{pmatrix}
\begin{array}{c}
\mbox{\scriptsize $(3\,5)$}\\[-1.5mm]
\longrightarrow\\[-1.2mm]
\mbox{\tiny Echelon}
\end{array}
{\mathcal{M}_C\,\, \longrightarrow \,\,}
\mathcal{R}_C=
\begin{pmatrix}
-1 & 2\\
-2/3 & 5/3 \\
7/3 & -4/3
\end{pmatrix}.
$$
Now we apply the algorithm described in section \ref{sec_diagonal5}. First, we need to choose a pair $j_3,j_4\in\{1,2,3\}$ such that the field $L=\Q(\alpha_{3,j_3},\alpha_{4,j_4})$ is a quadratic field or $\Q$. The only possible case is $(j_3,j_4)=(1,2)$ where $L=\Q(\sqrt{10})$. Next , we obtain $\widetilde{\mathfrak{S}}=\{\pm 1,\pm 2,\pm 3,\pm 6\}$. Now for any $\delta \in\widetilde{\mathfrak{S}}$, we must compute all the points $(t,w) \in
H_s^{\delta}(\Q(\sqrt{10}))$ with $t\in \mathbb{P}^1(\Q)$ for some  $s\in\{(\pm,\pm)\}$. We have obtained that for any $\delta \in\widetilde{\mathfrak{S}}$ there exists $s\in\{(+,\pm)\}$ such that $\rank_{\Z}H^{\delta}_{s}(\Q(\sqrt{10}))=1$ therefore we can apply the elliptic curve Chabauty method to obtain the possible values of $t$. The following table shows all the data that we have computed. The absolute value of the coordinates of the point $P\in C(\Q)$ for the corresponding $t$ appears at the last column:
$$
\begin{array}{|c|c|c|c|c|c|}
\hline
\delta & \mbox{signs} & H^{\delta}_{\mbox{\tiny signs}}(L)=\emptyset? & \rank_{\Z}H^{\delta}_{\mbox{\tiny signs}}(L) & t & P\\
\hline
-1 & (+,-) & \mbox{no} & 1 & 2 & [1: 1: 1: 1: 1]\\
\hline
1 & (+,-) & \mbox{no} & 1 & \infty & [1: 1: 1: 1: 1]\\
\hline
2 & (+,-) & \mbox{no} & 1 & -1 & [1:3:5:7:9]\\
\hline
-2 & (+,-) & \mbox{no} & 1 & 4/3 &[1:3:5:7:9]\\
\hline
3 & (+,+) & \mbox{no} & 1 & -2 & [1:3:5:7:9]\\
\hline
-3 & (+,+) & \mbox{no} & 1 & 3/2 & [1:3:5:7:9]\\
\hline
6 & (+,+) & \mbox{no} & 1 & 0 & [1: 1: 1: 1: 1]\\
\hline
-6 & (+,+) & \mbox{no} & 1 & 1 & [1: 1: 1: 1: 1]\\
\hline
\end{array}
$$
Looking at the previous table, we obtain 
$$
C(\Q)=\{[\pm 1:\pm 1:\pm 1:\pm 1:\pm 1], [\pm 1:\pm 3:\pm 5:\pm 7:\pm 9]\}.
$$ 
The unique non-trivial solution is $[\pm 1:\pm 3:\pm 5:\pm 7:\pm 9]$ that corresponds to $d=1$ and $m=0$. Therefore we obtain:

\

{\it \underline{Answer}: If $m$ and $d$ are integers with $d$ not a perfect square, then $Y = 1, 3, 5, 7, 9$ cannot be in arithmetic progression on the curve $X^2-dY^2 = m$.}

\subsection{Arithmetic progressions on Edwards curves}
An Edwards curve is an elliptic curve given in the form $E_d\,:\,x^2+y^2=1+dx^2y^2$, for some $d\in\Q$, $d\ne 0,1$. Let $y_n\in\Q$ such that $(n,y_n)\in E_d(\Q)$ for $n=0,\pm 1,\pm 2,\pm 3,\pm 4$. Then we say that  $(n,y_n)$ (or just $n$), $n=0,\pm 1,\dots,\pm 4$, are in arithmetic progression on $E_d$. For any $d$ we have that $(\pm 1,0),(0,\pm 1)\in E_d(\Q)$ therefore $y_0=0,y_{\pm 1}=\pm 1$. We can assume $n>1$ since if $(x,y)\in E_d(\Q)$ then $(\pm x, \pm y),(\pm y,\pm x)\in E_d(\Q)$. Now, denote by 
$$
d_n=\frac{n^2+y_n^2-1}{n^2 y_n^2}=\frac{z_n^2(n^2-1)+1}{n^2}\qquad \mbox{with}\,\,z_n=\frac{\pm 1}{y_n}.
$$
The existence of $d\in \Q$, $d\ne 0,1$ such that there exist $y_n\in\Q$ with $(n,y_n)\in E_d(\Q)$ for $n=0,\pm 1,\pm 2,\pm 3,\pm 4$ is characterized by $d_2=d_3$ and $d_2=d_4$. That is, by the diagonal genus $1$ curve defined by:
$$E
\,:\,
\left\{
\begin{array}{l}
5+27 z_2^2-32 z_3^2 =0\\
1+4z_2^2-5 z_4^2=0
\end{array}\right\}.
$$
This elliptic curve has Cremona reference \verb|33600es2| and has rank $2$. Then Moody \cite{Moody2011} proved that there are infinitely many Edwards curves with $9$ points in arithmetic progression. Then Moody said:

\

{\it \underline{Moody}: We performed a computer search to find a rational point on the curve E, leading to an $E_d$ with points having $x$-coordinates $\pm 5$. Our search has not found such a rational point, thus it is an open problem to find an Edwards curve with an arithmetic progression of length $10$ or longer}.

\

 Our first objective in \cite{Gonzalez-Jimenez2013} was to prove that there does not exist a rational $d$ such that $0,\pm 1,\dots,\pm 5$ form an arithmetic progression in $E_d(\Q)$. This objective was completed\footnote{Recently, Bremner \cite{Bremner2013} has obtained the same result but with a different proof.}. Here we show the details. Note, that in the paper \cite{Gonzalez-Jimenez2013} we studied the general case of arithmetic progressions of the form $a,a+q,\dots$ for any $a,q\in\Q$ on Edwards curves.

Now we impose $(\pm 5,y_{\pm 5})\in {E_d}(\Q)$, for some $y_{\pm 5}\in \Q$. This implies adding the equality $d_2=d_5$ to the system of equations: $\{d_2=d_3\,,\,d_2=d_4\}$. Therefore we obtain the genus $5$ curve:
\begin{equation}\label{moody}
C\,:\,
\left\{
\begin{array}{l}
5+27 z_2^2-32 z_3^2 =0\\
1+4z_2^2-5 z_4^2=0\\
7+25z_2^2-32 z_5^2=0
\end{array}\right\}.
\end{equation}
If we homogenize the equations (\ref{moody}) then the matrix corresponding to the squares of the variables is $\widehat{\mathcal{M}}_C$ and we can prove that $C$ is a diagonal genus $5$ curve computing its associated reduce matrix $\mathcal{R}_C$:
$$
\widehat{\mathcal{M}}_C=\begin{pmatrix}
1 & 4 & 0 & -5 & 0\\
7 & 25 & 0 & 0 & -32 \\
2 & 0 & 25 & 0 & -27
\end{pmatrix}
\begin{array}{c}
\mbox{\scriptsize $(1\,4)(2\,5)$}\\[-1.5mm]
\longrightarrow\\[-1.2mm]
\mbox{\tiny Echelon}
\end{array}
{\mathcal{M}_C\,\, \longrightarrow \,\,}
\mathcal{R}_C=
\begin{pmatrix}
1/5 & 4/5\\
7/32 & 25/32 \\
5/32 & 27/32
\end{pmatrix}.
$$
Let $(j_3,j_4)=(2,1)$ then $L=\Q(\alpha_{3,j_3},\alpha_{4,j_4})=\Q(\sqrt{-15})$ and $\widetilde{\mathfrak{S}}=\{\pm 1,\pm 2,\pm 3,\pm 6\}$. Then the following table shows all the data necessary to compute $C(\Q)$:
$$
\begin{array}{|c|c|c|c|c|c|}
\hline
\delta & \mbox{signs} & H^{\delta}_{\mbox{\tiny signs}}(L)=\emptyset? & \rank_{\Z}H^{\delta}_{\mbox{\tiny signs}}(L) & t & P\\
\hline
-1 & (+,+) & \mbox{no} & 1 & 4/5 & [1: 1: 1: 1: 1]\\
\hline
1 & (+,+) & \mbox{no} & 1 & \infty & [1: 1: 1: 1: 1]\\
\hline
2 & (+,+) & \mbox{no} & 1 &  \infty & [1: 1: 1: 1: 1]\\
\hline
-2 & (+,+) & \mbox{no} & 1 &  \infty &[1: 1: 1: 1: 1]\\
\hline
3 & (+,+) & \mbox{no} & 1 &  \infty & [1: 1: 1: 1: 1]\\
\hline
-3 & (+,-) & \mbox{no} & 1 &  \infty & [1: 1: 1: 1: 1]\\
\hline
6 & (+,-) & \mbox{no} & 1 & 0 & [1: 1: 1: 1: 1]\\
\hline
-6 & (+,-) & \mbox{no} & 1 & -1/5 & [1: 1: 1: 1: 1]\\
\hline
\end{array}
$$
That is, we obtain: 
$$
C(\Q)=\{[\pm 1:\pm 1:\pm 1:\pm 1:\pm 1]\}.
$$ 

\

{\it \underline{Answer}: There is no $d\in\Q$, $d\ne 0,1$, such that $0,\pm1,\dots,\pm 5$ form an arithmetic progression on an Edwards curve $E_d$.}
\subsection{Arithmetic progressions on elliptic curves in Weierstrass form}
Let $E$ be an elliptic curve given by a Weierstrass equation
$$
E\,:\,y^2+a_1xy+a_3y=x^3+a_2x^2+a_4x+a_6, \qquad a_1,a_2,a_3,a_4,a_6\in\Q.
$$ 
A set of rational points $P_1,\dots,P_n\in E(\Q)$ is said to be an arithmetic progression on $E$ of length $n$ if the $x$-coordinates form an arithmetic progression. Note that any two Weierstrass equation for an elliptic curve are related by a linear change of variables with $x$-coordinate of the form $x=u^2x'+r$. Therefore, an arithmetic progression on an elliptic curve given by a Weierstrass equation is independent of the Weierstrass model chosen. Thus, without loss of generality, we can work with short Weierstrass equation:
$$
E\,:\,y^2=x^3+Ax+B, \qquad A,B\in\Q.
$$ 
Let $a,q,Y_n\in\Q$, $n=0,\pm 1,\pm 2$ such that $(a+n q,Y_n)\in {E}(\Q)$, $n=0,\pm 1,\pm 2$. Then we have 
$$
\begin{array}{l}
Y_2^2=(a+2 q)^3+A(a+2 q)+B,\\[1mm]
Y_1^2=(a+q)^3+A(a+ q)+B,\\[1mm]
Y_0^2=a^3+Aa+B,\\[1mm]
Y_{-1}^2=(a- q)^3+A(a- q)+B,\\[1mm]
Y_{-2}^2=(a-2 q)^3+A(a-2 q)+B.\\[1mm]
\end{array}
$$
Bremner \cite{Bremner1999} reduced the previous system of $5$ equations to the following quadric in $\mathbb P^4$:
$$
- R^2 +4 S^2 -6 T^2 +4 U^2=V^2,
$$
where
$$
{\tiny
\left\{
\begin{array}{l}
\!\!\!\!a=6(S^2-2T^2+U^2),\qquad q=6(R^2-3S^2+3T^2-U^2),\\
\!\!\!\!A=-36(R^4-9R^2S^2+21S^4+6R^2T^2-39S^2T^2+21T^4+R^2U^2+6S^2U^2-9T^2U^2+U^4),\\
\!\!\!\!B=216(R^4S^2-9R^2S^4+20S^6+4R^4T^2-12R^2S^2T^2-21S^4T^2+24R^2T^4-21S^2T^4\\
+20T^6+R^4U^2-8R^2S^2U^2+24S^4U^2-8R^2T^2U^2-12S^2T^2U^2-9T^4U^2+R^2U^4+4S^2U^4+T^2U^4).
\end{array}\right.
}$$
Bremner parametrizes the quadric above obtaining:
$$
\left\{
\begin{array}{l}
R = w^2 - 8 w x + 12 w y - 88 w z + 4 x^2 - 6 y^2 + 4 z^2,\\
S = -w^2 + 2 w x + 12 x y - 4 x^2 - 8 x z - 6 y^2 + 4 z^2,\\
T = -w^2 + 2 w y + 4 x^2 - 8 x y + 6 y^2 - 8 y z + 4 z^2,\\
U = -w^2 + 2 w z + 4 x^2 - 8 x z - 6 y^2 + 12 y z - 4 z^2,\\
V = -w^2 + 4 x^2 - 6 y^2 + 4 z^2.\\
\end{array}
\right.
$$
%
Now we impose $(a\pm 3 q,Y_{\pm 3})\in {E}(\Q)$, for some $Y_{\pm 3}\in \Q$. This implies:
\begin{equation}\label{bremner3}
4 R^2 - 6 S^2 + 4 T^2 - U^2=V_1^2\,,\qquad-4 R^2 + 15 S^2 - 20 T^2 + 10 U^2=V_2^2.
\end{equation}
The equations (\ref{bremner3}) define a variety $\mathcal V$ of dimension $3$ in $\mathbb P^5$. Elliptic curves on Weierstrass form over $\Q$ with $7$ points in arithmetic progressions are characterized by the rational point of a variety of dimension $3$, which is still an intractable problem nowadays. Nevertheless, Bremner noticed that if we intersect this variety with the one with equations $w=x$ and $z=0$, we obtain the solution to  (\ref{bremner3})  that gives:
$$
\begin{array}{l}
(a,q)=(0,6xy(x-y)(x-2y))\,,\\
(A,B)=x^2y^2(x-y)^2(x-2y)^2(-252,324(x^2-2xy+2y^2)^2).
\end{array}
$$
Now, with the restrictions above, we impose $(a\pm 4 q,Y_{\pm 4})\in {E}(\Q)$, for some $Y_{\pm 4}\in \Q$. This implies:
\begin{equation}\label{bremner4}
\left\{
\begin{array}{l}
z^2=x^4 + 20 x^3y - 64x^2y^2 + 40xy^3 + 4y^4\\
w^2=x^4 - 28 x^3y + 80 x^2y^2 - 56 xy^3 + 4y^4
\end{array}
\right\},
\end{equation}
for some $w,z\in \Q$. Bremner checked that each equation on (\ref{bremner4}) corresponds to the elliptic curve with Cremona reference \verb|840e2| that has rank $1$ and therefore he built a infinite family of elliptic curve on Weiersstrass form with $8$ points in arithmetic progression. Nevertheless, he could not prove if there are $9$ points in arithmetic progression in his family of elliptic curves. Then Bremner asserted:

\

{\it \underline{Bremner:} \lq \lq For nine points in the arithmetic progression, it is necessary to satisfy (\ref{bremner4}) simultaneously, and this corresponds to determining rational points on a curve of genus $5$. There are only finitely many such points, and it seems plausible that they are given by $\pm(x,y)=(1,0),(0,1),(1,1),(2,1)$ (each leading to degenerate progressions) but we are unable to verify this"}. 

\

Now, our objective in this section is to verify the previous assumption. Let us denote by $p_1(t)=t^4+20t^3-64t^2+40t+4$, $p_2(t)=t^4-28t^3+80t^2-56t+4$ and 
$$
C:\{z_1^2=p(t)\,,z_2^2=q(t)\}.
$$
Therefore to compute all solutions to (\ref{bremner4}) is equivalent to computing $C(\Q)$. Then we apply the algorithm from section \ref{sec2}. Both polynomials $p_1$ and $p_2$ factorize over the same quadratic fields: $\Q(\sqrt{30})$, $\Q(\sqrt{35})$, $\Q(\sqrt{42})$. Notice that $\delta=1$ always belongs to ${\mathfrak{S}}$. Let $L=\Q(\sqrt{D})$, for $D\in\{30,35,40\}$, then $\rank_{\Z}H^{1}_{(\pm,\pm)}(L)>1$. Therefore we can not apply the elliptic curve Chabauty method and our algorithm does not compute $C(\Q)$. Nevertheless, we can check that in fact $C$ is diagonal. We have the relations:
$$
\left\{
\begin{array}{l}
p(t) + q(t) =2 (2-2t+t^2)^2\\
7 p(t) + 5 q(t) =12(-2+t^2)^2\\
5 p(t) + 7 q(t)=12(2-4t+t^2)^2\\
\end{array}
\right\}.
$$
That is, another model for the curve $C$ is
$$
C\,:\,\left\{
\begin{array}{l}
z_1^2+z_2^2=2z_3^2\\
7z_1^2+5z_2^2=12z_4^2\\
5 z_1^2+7z_2^2=12 z_5^2
\end{array}
\right\}.
$$
Then we can apply the algorithm from section \ref{sec_diagonal5}. In our case we have 
$$
\widehat{\mathcal{M}}_C=\begin{pmatrix}
1 & 1 & -2 & 0 & 0\\
7 & 5 & 0 & -12 & 0 \\
5 & 7 & 0 & 0 & -12\end{pmatrix}
\begin{array}{c}
\mbox{\scriptsize $(2\,4\,3\,5)$}\\[-1.5mm]
\longrightarrow\\[-1.2mm]
\mbox{\tiny Echelon}
\end{array}
{\mathcal{M}_C\,\, \longrightarrow \,\,}
\mathcal{R}_C=
\begin{pmatrix}
-1 & 2\\
1/6 & 5/6 \\
-1/6 & 7/6
\end{pmatrix}.
$$
Let be $(j_3,j_4)=(3,1)$, then $L=\Q(\alpha_{3,j_3},\alpha_{4,j_4})=\Q(\sqrt{7})$  and  $\widetilde{\mathfrak{S}}=\{1,2,3,6\}$. Then the following table shows all the data necessary to compute $C(\Q)$ and the solutions of (\ref{bremner4}):
$$
\begin{array}{|c|c|c|c|c|c|c|}
\hline
\delta & \mbox{signs} & H^{\delta}_{\mbox{\tiny signs}}(L)=\emptyset? & \rank_{\Z}H^{\delta}_{\mbox{\tiny signs}}(L) & t & P & \pm (x,y)\\
\hline
1 & (+,-) & \mbox{no} & 1 & \begin{array}{c}1\\ \infty\end{array} & [1,1,1,1,1] & \begin{array}{c}(1,1)\\ (1,0)\end{array} \\
\hline
2 & (+,+) & \mbox{no} & 1 & \infty & [1,1,1,1,1] & (1,0)\\
\hline
3 & (+,+) & \mbox{yes} & - & - & - & -\\
\hline
6 & (+,+) & \mbox{no} & 1 &\begin{array}{c}0\\ 2\end{array} & [1,1,1,1,1] & \begin{array}{c}(0,1)\\ (2,1)\end{array}\\
\hline
\end{array}
$$
Looking at the previous table, we obtain 
$$
C(\Q)=\{[\pm 1,\pm 1,\pm 1,\pm 1,\pm 1]\},
$$ 
which allows us to prove the following:

\

{\it \underline{Fact}:} There are no nine points in arithmetic progression on the family of elliptic curves
$$
E\,:\,Y^2=X^3+AX+B,\,\,\,
\left\{
\begin{array}{l}
A=-252x^2y^2(x-y)^2(x-2y)^2,\\[2mm]
B=324x^2y^2(x-y)^2(x-2y)^2(x^2-2xy+2y^2)^2.
\end{array}
\right.
$$
\subsection{$\Q$-derived polynomials}
A univariate polynomial $p(x)\in\Q[x]$ is called $\Q$-derived if $p(x)$ and all its derivatives split completely over $\Q$ (i.e. all their roots belong to $\Q$). Note that if $q(x)$ is $\Q$-derived then for any $r,s,t\in\Q$, the polynomial $rq(sx+t)$ is $\Q$-derived too. Therefore a relation between $\Q$-derived polynomials is established: two $\Q$-derived polynomial $p(x)$ and $q(x)$ are equivalent if $q(x)=rp(sx+t)$ for some $r,s,t\in\Q$. Buchholz and MacDougall considered the problem to classifying all $\Q$-derived polynomials up to the above relationship in \cite{Buchholz-MacDougall2000}:

\

\noindent{\bf Conjecture. }{\em All $\Q$-derived polynomials are equivalent to one of the following:
$$
x^n, x^{n-1}(x-1),x(x-1)\!\!\left(x-\frac{v(v-2)}{v^2-1}\right)\!\!,x^2(x-1)\!\!\left(x-\frac{9(2w+z-12)(w+2)}{(z-w-18)(8w+z)}\right)
$$
for some $n\in\Z$, $v\in\Q$, $(w,z)\in E(\Q)$ where $E:z^2=w(w-6)(w+18)$.
}

\

A polynomial is of type $p_{m_1,\dots,m_r}$ if it has $r$ distinct roots and $m_i$ is the multiplicity of the $i$-th root. Buchholz and MacDougall  \cite{Buchholz-MacDougall2000} proved the above conjecture under the two hypotheses: non existence of  $\Q$-derived polynomials of type $p_{3,1,1}$ and $p_{1,1,1,1}$.  

\subsubsection{$\Q$-derived polynomials of type $p_{3,1,1}$} Let $q(x)$ be a $\Q$-derived polynomial of type $p_{3,1,1}$. Then without loss of generality we can assume that $q(x)=x^3(x-1)(x-a)$ for some $a\in \Q$ with $a\ne 0,1$. Moreover, the discriminants of the quadratic polynomials $q'''(x),q''(x)/x$ and $q'(x)/x^2$ are all squares over $\Q$ (cf. \cite[\S 2.3]{Buchholz-MacDougall2000}).  That is, there exist $b_1,b_2,b_3\in\Q$ such that
$$
b_1^2=4a^2-2a+4\,,\qquad b_2^2=9a^2-12a+9\,,\qquad b_3^2=4a^2-7a+4.
$$
Now, changing $a=(X-3)/(X+3)$ and $b_i=Y_i/(X+3)^3$ for $i=1,2,3$, we obtain the equivalent problem
\begin{equation}\label{flynn}
Y_1^2=6(X^2+15)\,,\qquad Y_2^2=6(X^2+45)\,,\qquad Y_3^2=X^2+135,
\end{equation}
where $X,Y_1,Y_2,Y_3\in \Q$. Flynn \cite{Flynn2001} proved that the unique solutions to (\ref{flynn}) are $(X,Y_1,Y_2,Y_3)=(\pm 3,\pm 12,\pm 18,\pm 12)$, proving that no polynomial of type $p_{3,1,1}$ is $\Q$-derived.

Here we give a different proof based on the algorithm described in section \ref{sec2}. Note that (\ref{flynn}) defines a diagonal genus $5$ curve $C$ with model of the form (\ref{model}) and associated matrix
$$
\mathcal{M}_C=\begin{pmatrix}1 & 0 & 0 & -6 & -90  \\ 0 & 1 & 0 & -6 & -270 \\  0 & 0 & 1 & -1& -135\end{pmatrix}
\begin{array}{c}
\mbox{\scriptsize $(1\, 2\, 5)$}\\[-1.5mm]
\longrightarrow\\[-1.2mm]
\mbox{\tiny Echelon}
\end{array}
\mathcal{R}_C=
\begin{pmatrix}
-1/45 & 1/270\\
4 & 1/3 \\
-2 & 1/2
\end{pmatrix}
$$
Let us apply the algorithm described in section \ref{sec_diagonal5}. In this case we have that $\Q(\alpha_{4,1})=\Q$; therefore for any choice of $j_3$ we have that $L=\Q(\alpha_{3,j_3},\alpha_{4,1})$ has degree less than or equal $2$. We use $j_3=3$ where $L=\Q(\sqrt{5})$ and $\widetilde{\mathfrak{S}}=\{1, 2, 3, 6\}$. The following table shows all the data necessary to compute $C(\Q)$:
$$
\begin{array}{|c|c|c|c|c|c|}
\hline
\delta & \mbox{signs} & H^{\delta}_{\mbox{\tiny signs}}(L)=\emptyset? & \rank_{\Z}H^{\delta}_{\mbox{\tiny signs}}(L) & t & P\\
\hline
1 & (+,+) & \mbox{no} & 1 & 0,\infty & [3,12,18,12,1]\\
\hline
2 & (+,-) & \mbox{no} & 1 & 8/15,16/5 & [3,12,18,12,1]\\
\hline
3 & (+,+) & \mbox{yes} & - & - & -\\
\hline
6 & (+,+) & \mbox{yes} & - & - & -\\
\hline
\end{array}
$$
Looking at the previous table, we obtain 
$$
C(\Q)=\{[\pm 3,\pm 12,\pm 18,\pm 12,\pm 1]\}.
$$ 
\subsubsection{$\Q$-derived polynomials of type $p_{1,1,1,1}$} In this case, with similar ideas as the previous case, it may be proved (cf. \cite[\S 2.2.3]{Buchholz-MacDougall2000}) that without loss of generality a polynomial of type $p_{1,1,1,1}$ can be assumed to be of the form
$$
p(x)=(x-1)(x-a)(x-b)\left(x-\frac{-a b}{a+b+ab}\right)
$$
with $a,b\in\Q$,  $a,b\ne 1$ and $a\ne b$. Furthermore, there must exist $z,w\in\Q$ such that 
\begin{equation}\label{surfaceQ}
z^2=r_4b^4-r_3b^3+r_2b^2+r_1b+r_0\,,\qquad w^2=s_4b^4-s_3b^3+s_2b^2+s_1b+s_0
\end{equation}
where 
$$
\begin{array}{lccl}
r_4=9a^2+18a+9, & & &s_4=9a^2+18a+9,\\
r_3=14a^3+10a^2+10a+14, & & & s_3=6a^3-6a^2-6a+6,\\
r_2=9a^4-10a^3-6a^2-10a+9,& & & s_2=9a^4+6a^3+18a^2+6a+9,\\
r_1=18a^4-10a^3-10a^2+18a, & & & s_1=18a^4+6a^3+6a^2+18a,\\
r_0=9a^4-14a^3+9a^2, & & & s_0=9a^4-6a^3+9a^2.
\end{array}
$$
That is, Buchholz and MacDougall  \cite[\S 2.2.3]{Buchholz-MacDougall2000} gave a characterization of $\Q$-derived polynomial of type $p_{1,1,1,1}$ in terms of rational points on the surface\footnote{A similar characterization has been done by Stroeker \cite{Stroeker2006}.} $S$  on $\mathbb{P}^4$ defined by (\ref{surfaceQ}). Note that $S$ could be considered as a genus $5$ curve over the field $\Q(a)$. Therefore, if we fix $a\in\Q$ and we denote by $S_a$ the corresponding genus $5$ curve, we may apply the algorithm described in section \ref{sec2} to compute $S_a(\Q)$.

\appendix
\section{On quartic elliptic curves}
\subsection{Rational points}
Let $q(t)$ be a monic quartic separable polynomial in $K[t]$. Then the equation $y^2=q(t)$ defines a genus $1$ curve, which we call $F$. The purpose of this section is to give a method that allows to compute the set of points $F(K)$. This method is Proposition 14 from \cite{Gonzalez-Jimenez-Xarles2013b}. We include its statement and proof (due to Xavier Xarles) for the sake of completeness:
\begin{proposition}\label{2cov}
Let $F$ be a genus $1$ curve over a number field $K$ given by a
quartic model of the form $y^2=q(t)$, where $q(t)$ is a monic quartic polynomial in $K[t]$. Thus, the curve $F$ has two
rational points at infinity, and we fix an isomorphism from $F$ to
its Jacobian $E=\Jac(F)$ defined by sending one of these points
at infinity to $\cO$, the zero point of $E$. Then:

(1) Any $2$-torsion point $P\in E(K)$
corresponds to a factorization $q(t)=q_1(t)q_2(t)$, where $q_1(t),q_2(t) \in L[t]$ quadratics and $L/K$ is an algebraic extension of degree at most $2$.

(2) Given such a $2$-torsion point $P$, the degree two
unramified covering $\chi:F'\rightarrow F$ corresponding to the degree two
isogeny $\phi:E'\to E$ determined by $P$ can be described as the
map from the curve $F'$ defined over $L$, with affine part in
$\mathbb{A}^3$ given by the equations $y_1^2=q_1(t)$ and
$y_2^2=q_2(t)$ and the map given by $\chi(t,y_1,y_2)=(t,y_1y_2)$.

(3) Given any degree two isogeny $\phi:E'\to E$, consider the
Selmer group $\Sel(\phi)$ as a subgroup of $K^*/(K^*)^2$. Let
$\cS_L(\phi)$ be a set of representatives in $L$ of the image
of $\Sel(\phi)$ in $L^*/(L^*)^2$ via the natural map. For any
$\delta\in \cS_L(\phi)$, define the curve $F'^{(\delta)}$ given
by the equations $\delta y_1^2=q_1(t)$ and $\delta y_2^2=q_2(t)$,
and the map to $F$ defined by
$\chi^{(\delta)}(t,y_1,y_2)=(t,y_1y_2\delta)$. Then
$$F(K) \subseteq \bigcup_{\delta \in \cS_L(\phi)} \chi^{(\delta)}(\{(t,y_1,y_2)\in F'^{(\delta)}(L)\ : \ t\in \mathbb{P}^1(K)\}).$$
\end{proposition}

\begin{proof}(Xarles)
First we prove (1) and (2). Suppose we have such a factorization
$q(t)=q_1(t)\,q_2(t)$ over some extension $L/K$, with $q_1(t)$ and
$q_2(t)$ monic quadratic polynomials. Then the covering
$\chi:F'\rightarrow F$ from the curve $F'$ defined over $L$, with
affine part in $\mathbb{A}^3$ given by the equations $y_1^2=q_1(t)$ and $y_2^2=q_2(t)$ and the map given by
$\chi(t,y_1,y_2)=(t,y_1y_2)$, is an unramified degree two
covering. So $F'$ is a genus 1 curve, and clearly it contains the
preimage of the two points at infinity, which are rational over
$L$, hence it is isomorphic to an elliptic curve $E'$. Choosing
such isomorphism by sending one of the preimages of the fixed
point at infinity to $\cO$, we obtain a degree two isogeny $E'\to E$,
which corresponds to a choice of a two torsion point.

So, if the polynomial $q(t)$ decomposes completely in $K$, the
assertions (1) and (2) are clear since the number
of decompositions $q(t)=q_1(t)\,q_2(t)$ as above is equal to the
number of points of exact order 2. Now the general case is proved
by Galois descent: a two torsion point $P$ of $E$ is defined over
$K$ if and only if the degree two isogeny $E'\to E$ is defined
over $K$, so if and only if the corresponding curve $F'$ is
defined over $K$. Hence the polynomials $q_1(t)$ and $q_2(t)$
should be defined over an extension of $L$ of degree $\le 2$, and in
case they are not defined over $K$, the polynomials $q_1(t)$ and
$q_2(t)$ should be Galois conjugate over $K$.

Now we show the last assertion. First, notice that the curves
$F'^{(\delta)}$ are twisted forms (or principal homogeneous
spaces) of $F'$, and it becomes isomorphic to $F'$ over the
quadratic extension of $L$ adjoining the square root of $\delta$.

Consider the case where $L=K$. So $F'$ is defined over $K$. For any
$\delta\in \Sel(\phi)$, consider the associated homogeneous space
$D^{(\delta)}$; it is a curve of genus $1$ along with a degree
$2$ map $\phi^{(\delta)}$ to $E$, without points in any local
completion, and isomorphic to $E'$ (and compatible with $\phi$)
over the quadratic extension $K(\sqrt{\delta})$. Moreover, it is
determined by such properties (see \cite[\S 8.2]{Cohen239}). So,
by this uniqueness, it must be isomorphic to $F'^{(\delta)}$
along with $\chi^{(\delta)}$. The last assertion also is clear
from the definition of the Selmer group.

Now, the case $L\ne K$. The assertion is proved just
observing that the commutativity of the diagram
$$\xymatrix{
  \Sel(\phi) \ar@{->}[r]    \ar@{->}[d]& \Sel(\phi_L)   \ar@{->}[d]\\
   K^*/(K^*)^2 \ar@{->}[r]& L^*/(L^*)^2 
  }
$$
where the map $\Sel(\phi)\to \Sel(\phi_L)$ is the one sending the
corresponding homogeneous space to its base change to $L$.
\end{proof}

\subsection{A Galois theory exercise on quartic polynomials}\label{sec_fact}

We show an algorithm to factorize a quartic polynomial as a product of two quadratic polynomials over an extension of degree at most two.

Let be a quartic polynomial $p(t)=t^4 + a t^3 + b t^2 + c t + d$ over a number field $K$, and its factorization given by
$$
p(t)=(t-\alpha_1)(t-\alpha_2)(t-\alpha_3)(t-\alpha_4),
$$
over an algebraic closure $\overline{K}$. Then all the factorizations of $p(t)$ as product of two quadratic polynomials are of the form $p(t)=p_1(t)p_2(t)$ where:
$$
p_1(t)=(t^2-(\alpha_1+\alpha_2)t+\alpha_1\alpha_2)\quad\mbox{and}\quad p_2(t)=(t^2-(\alpha_3+\alpha_4)t+\alpha_3\alpha_4).
$$
There are three polynomials related to a quartic polynomial that are of great utility for the study of the Galois group of the quartic polynomial $p(t)$. These are the cubic resolvent of $p(t)$:
$$
r(t)=t^3 -b t^2 + (a c - 4 d) t -a^2 d + 4 b d - c^2,
$$
and if $\beta\in K$ is a root of $r(t)$, define 
$$
\begin{array}{lcl}
r_1(t)=t^2-\beta t+d, &\quad& \Delta_1=disc_t(r_1)=\beta^2-4 d,\\
r_2(t)=t^2+at+(b-\beta), &\quad& \Delta_2=disc_t(r_2)= 4 \beta+a^2-4 b.
\end{array}
$$

\begin{lemma}
If $\Delta_2\ne 0$ then $p_1(t),p_2(t)\in K(\sqrt{\Delta_2})[t]$. Otherwise, $p_1(t),p_2(t)\in K(\sqrt{\Delta_1})[t]$.
\end{lemma}
\begin{proof}
First suppose $\Delta_2\ne 0$. Let be $\gamma=\alpha_1+\alpha_2-(\alpha_3+\alpha_4)$. Then $\gamma^2 = \Delta_2$. Define
$$
\begin{array}{l}
f(t)=\frac{1}{2}(t-a),\\[3mm]
g(t)=\frac{1}{8}(4 b - a^2+\frac{2(a^4 - 6 a^2 b + 8 b^2 + 4 a c -  32 d)}{a^3 - 4 a b + 8 c} x + x^2 - \frac{3 a^2 - 8 b}{a^3 - 4 a b + 8 c} x^3 + \frac{1}{a^3 - 4 a b + 8 c} x^5),
\end{array}
$$ 
then 
$$
\alpha_1+\alpha_2=f(\gamma), \quad \alpha_1\alpha_2=g(\gamma), \quad\alpha_3+\alpha_4=f(-\gamma), \quad\alpha_3\alpha_4=g(-\gamma).
$$
That is, $p_1(t),p_2(t)\in K(\gamma)[t]=K(\sqrt{\Delta_2})[t]$.

Now, assume $\Delta_2=0$ and let be $\delta=(\alpha_2-\alpha_3)(\alpha_2-\alpha_4)$. Then $\delta^2 = \Delta_1$ and we have
$$
\alpha_1+\alpha_2=\alpha_3+\alpha_4=-\frac{a}{2}, \quad \alpha_3\alpha_4=\frac{c}{a}+\frac{\delta}{2},\quad \alpha_1\alpha_2=\frac{d}{\alpha_3\alpha_4}.
$$ 
That is, $p_1(t),p_2(t)\in K(\delta)[t]=K(\sqrt{\Delta_1})[t]$.
\end{proof}

\begin{remark} There is a nice relationship between the elliptic curves defined by the quartic $p(t)$ and the cubic $-r(-x)$ such that the lemma above could be obtained. Let us denote by
$$
\begin{array}{l}
\displaystyle F\,:\,v^2=p(u)=u^4 + a u^3 + b u^2 + c u + d=\prod_{i=1}^4(u-\alpha_i).\\
\displaystyle E\,:\,y^2=-r(-x)=x^3 +b x^2 + (a c - 4 d) x +a^2 d - 4 b d + c^2 = \prod_{j=2}^4(x+\delta_j),
\end{array}
$$
$\delta_i=\alpha_1\alpha_i+\alpha_j\alpha_k$ such that $\{1,2,3,4\}=\{1,i,j,k\}$. Then, there exists an isomorphism $\phi:F\longrightarrow E$  defined over $\Q$. Now, let us denote $\gamma_i=\alpha_1+\alpha_i-\alpha_j-\alpha_k$ for $\{1,i,j,k\}=\{1,2,3,4\}$. Assume that $\gamma_i\ne 0$ for $i=2,3,4$, then we have
$$
\begin{array}{c}
\displaystyle \phi([1:1:0])=[0:1:0]\,,\quad \phi([1:-1:0])=\left(\frac{1}{4}s_1^2-s_2,\frac{1}{8}\delta_2\delta_3\delta_4  \right),\\
\displaystyle \phi(\alpha_i,0)=\left(\alpha_i\left(\alpha_i-\sum_{i\ne j}{\alpha_j}\right),\prod_{j\ne i}(\alpha_i-\alpha_j)\right) ,
\end{array}
$$
where $s_k$ denote the symmetric polynomial of degree $k$ on $\alpha_1,\dots,\alpha_4$. Moreover, $\phi(\alpha_i,0)=\phi(\alpha_1,0)+(-\delta_j,0)$ for $j=2,3,4$.

Now, for the inverse we have:
$$
\displaystyle \phi^{-1}(-\delta_i,0)=\left(\frac{g(\gamma_i)-g(-\gamma_i)}{\gamma_i},\frac{\displaystyle -\prod_{j\ne i}{(\delta_i-\delta_j)}}{\gamma_i^2}\right).
$$

Let us move the point $(-\delta_i,0)$ to $(0,0)$. We obtain a new Weierstrass equation $W_i\,:\,y^2=x(x^2+A_ix+B_i)$ where
$$
A_i=-2\delta_i+\sum_{j\ne i}{\delta_j}\qquad \mbox{and} \quad B_i=\prod_{j\ne i}{(\delta_i-\delta_j)}.
$$
If we denote by $\psi_i$ the isomorphism between $F$ and $W_i$ and by $(x_i,y_i)=\psi^{-1}(0,0)$ we obtain the equalities
$$
\frac{-B_i}{y_i}=\gamma_i^2=disc_t(t^2+at+(b-\delta_i)),
$$
the second one coming from the lemma above.

Finally, let us assume that $\gamma_i=0$ for some $i\in\{2,3,4\}$. For simplicity, let $i=2$. In this particular case we have:
$$
\begin{array}{lcl}
\phi([1:1:0])=[0:1:0]\,, & & \phi([1:-1:0])=(-\delta_2,0)\,,\\
\displaystyle \phi(\alpha_1,0)=\left(-2\alpha_1\alpha_2,-\prod_{j\ne 2}(\alpha_2-\alpha_j)\right)\,, & &  \phi(\alpha_2,0)=-\phi(\alpha_1,0)\,,\\
\displaystyle \phi(\alpha_3,0)=\left(-2\alpha_3\alpha_4,-\prod_{j>k,k\ne 1}(\alpha_k-\alpha_j)\right)\,, & &  \phi(\alpha_4,0)=-\phi(\alpha_3,0)\,,\\
\end{array}
$$
and for the inverse
$$
\begin{array}{ll}
\phi^{-1}(-\delta_2,0)=[1:-1:0]\,,&\\
\displaystyle\phi^{-1}(-\delta_3,0)=\left(\frac{\alpha_3+\alpha_4}{2},\frac{1}{4}(\alpha_2-\alpha_1)(\alpha_3-\alpha_4)\right),  & \phi^{-1}(-\delta_4,0)=-\phi^{-1}(-\delta_3,0).
\end{array}
$$
Now move the point $(-\delta_2,0)$ to $(0,0)$ and obtain a new Weierstrass equation $W_2\,:\,y^2=x(x^2+A_2x+B_2)$ where
$$
A_2=\alpha_3^2+\alpha_4^2-2\alpha_1\alpha_2\qquad \mbox{and} \quad B_2=(\alpha_2-\alpha_3)^2(\alpha_2-\alpha_4)^2.
$$
Then if we denote by $\psi_2$ the isomorphism between $F$ and $W_2$ and by $[x_2:y_2:z_2]=\psi^{-1}(0,0)=[1:-1:0]$ we obtain the equalities
$$
\frac{-B_2}{y_2}=(\alpha_2-\alpha_3)^2(\alpha_2-\alpha_4)^2=disc_t(t^2-\delta_2t +d),
$$
the second one coming from the lemma above.
\end{remark}

\subsection{Diagonal genus $1$ curve}\label{sec_diagona_genus1}
%
%
%

Let $K$ be a number field and $a,b,c,d\in K$ such that $ad-bc\ne 0$. Then the matrix 
$$\mathcal{R}=\begin{pmatrix}a & b \\c & d\end{pmatrix}$$
defines the genus $1$ curve $C$ (that we call diagonal) given by  the intersection of the following two quadrics in $\mathbb P^3$:
\begin{equation}\label{diagonal1}
C\,:\,\left\{\begin{array}{c}aX_0^2+bX_1^2=X_2^2\\ cX_0^2+dX_1^2=X_3^2\end{array}\right\}. 
\end{equation}
Suppose that there exists  $P_0=[ x_0: x_1: x_2: x_3] \in C(K)$, then $C$ is an elliptic curve and it has a Weierstrass equation. Parametrizing the first conic of $C$ by the point $P_0$ obtaining
$$
\begin{array}{rcl}
[X_0:X_1:X_2]&=&[-a\,b\,x_0\, x_3^4-2\, b\, x_1\, x_3^2 t+x_0\, t^2 :\\
& & \qquad a\, b\, x_1\, x_3^4-2\, a\, x_0\, x_3^2\, t-x_1\, t^2 :  x_2\, (a\, b\, x_3^4+t^2)]
\end{array}
$$
with inverse given by $t=\frac{b (x_1 X_2 - X_1 x_2) x_3^2}{x_0 X_2 - X_0 x_2}$. Next, we substitute $X_0,X_1,X_2$ in
the second equation and we obtain the quartic $F:z^2=p_{\mathcal R}(t)$, where 
\begin{equation}\label{eq_quartic_diagonal_genus1}
\begin{array}{l}
p_{\mathcal{R}}(t)=p(t)=t^4+4(ad-bc)x_0x_1t^3+2(2(a^2dx_0^2+b^2cx_1^2)-abx_3^2))x_3^2t^2-\\ \qquad\qquad\qquad\quad-4ab(ad-bc)x_0x_1x_3^4t+a^2b^2x_3^8,
\end{array}
\end{equation}
and $(x_3 X_3)^2=p(t)$. Then the quartic genus $1$ curve $F$ has the Weierstrass equation:
$$
E\,:\,y^2=x(x+a\,d)(x+b\,c).
$$

The trivial points $[\pm x_0: \pm x_1:\pm x_2:\pm x_3] \subseteq {C}(K)$ goes to $\{Q_i\,:\,i=0\dots 7\}\subseteq F(K)$ and then to $\{P_i\,:\,i=0\dots 7\}\subseteq E(K)$:\\
$$
\begin{array}{|c|c|c|c|}
\hline
i & T_i &  Q_i & P_i \\[1mm]
\hline
0 & [++++]  & [0:1:0]  & \cO:=[0:1:0] \\[1mm]
\hline
1 & [--++]   & (0,a\,b\,x_3^4) & (0,0)\\[1mm]
\hline
2 & [-++-]   & \left(-a\,\frac{x_0x_3^2}{x_1},-a\,\frac{x_2^2 x_3^4}{x_1^2}\right) & (-b\,c,0 )\\[1mm]
\hline
3 & [-+-+]  & \left(b\,\frac{x_1x_3^2}{x_0},-b\,\frac{x_2^2 x_3^4}{x_0^2}\right) & (-a\,d,0 )\\[1mm]
\hline
4 & [++-+]   & (0,-a\,b\,x_3^4)& \left(-a\,b\frac{x_3^2}{x_2^2},a\,b(a\,d-b\,c)\frac{x_0 x_1 x_3}{x_2^3}\right)\\[1mm]
\hline
5 & [+-++]   & \left(b\,\frac{x_1x_3^2}{x_0},b\,\frac{x_2^2 x_3^4}{x_0^2}\right) & \left(b\,d\frac{x_1^2}{x_0^2},-b\,d\frac{x_1 x_2 x_3}{x_0^3}\right)\\[1mm]
\hline
6 & [-+++]   & \left(-a\,\frac{x_0 x_3^2}{x_1},a\,\frac{x_2^2 x_3^4}{x_1^2}\right) & \left(a\,c\frac{x_0^2}{x_1^2},a\,c\frac{x_0 x_2 x_3}{x_1^3}\right) \\[1mm]
\hline
7 & [+++-]   & [1:-1:0] & \left(-c\,d\frac{x_2^2}{x_3^2},-c\,d(a\,d-b\,c)\frac{x_0 x_1 x_2}{x_3^3}\right)\\[1mm]
\hline
\end{array}
$$

\

Note that the set $\{P_i\,:\,i=0\dots 7\}$ is generated by $P_2,P_3,P_4$ and, in particular, $\Z/2\Z\otimes\Z/2\Z\oplus\langle P_4\rangle$ is a subgroup of $E(K)$. Therefore, the rank of the Mordell-Weil group of $E(K)$ is, in general, non-zero.

\

Now, in section \ref{sec_fact} we have described a method to factorize a quartic polynomial as the product of two quadratic polynomials over a quadratic field. Applying this method to the polynomial $p(t)$ we obtain the factorization $p(t)=p_{i+}(t){p_{i-}}(t)$ over $\Q(\alpha_i)$ corresponding to the $2$-torsion point $P_i$, for $i=1,2,3$:
$$
\begin{array}{ll}
\begin{array}{l}
p_{1+}(t)=t^2+2((a\,d-b\,c)x_0x_1-x_2^2\alpha_1)t-a\,bx_3^4,
\end{array}  & \alpha_1=\sqrt{-c\,d}\\[2mm]
\begin{array}{l}
p_{2+}(t)=t^2+2\,x_0((a\,d-b\,c)x_1-x_2\alpha_2 )t\\[1mm]
\qquad\quad\quad\quad +\,b x_0^2(a\,cx_0^2+(2\,b\,c-a\,d)x_1^2+2x_1r_0\alpha_2),
\end{array}  & \alpha_2=\sqrt{-c(a\,d-b\,c)}\\\\[2mm]
\begin{array}{l}p_{3+}(t)=t^2+2x_1((a\,d-b\,c)x_0-x_2\alpha_3)t\\[1mm]
\qquad\quad\quad\quad+\,ax_3^2(b\,d\,x_1^2+(2\,a\,d-b\,c)x_0^2-2\,x_0 x_2\alpha_3), 
\end{array}
 & \alpha_3=\sqrt{d(a\,d-b\,c)}
\end{array}
$$
and ${p_{i-}}(t)$ is obtained replacing $\alpha_i$ by $-\alpha_i$ on $p_i(t)$.

\

{\bf Acknowledgements:} 
We would like to thank to Xavier Xarles by his continuous inspiration, without his ideas this paper would have not been possible; and José M. Tornero, who read the earlier versions of this paper carefully. Finally, the author thanks the anonymous referee for useful comments.

\end{document}